\documentclass[12pt]{amsart}  
\usepackage{amsfonts,amsmath,amsthm,amscd,amssymb,latexsym}
\usepackage{latexsym}\usepackage[all]{xy}
\swapnumbers

\newcounter{zlist}
\newenvironment{zlist}{\begin{list}{\rm(\arabic{zlist})}{
\usecounter{zlist}\leftmargin2.5em\labelwidth2em\labelsep0.5em
\topsep0.6ex\itemsep0.3ex plus0.2ex minus0.3ex
\parsep0.3ex plus0.2ex minus0.1ex}}{\end{list}}
 
\newcounter{blist}
\newenvironment{blist}{\begin{list}{{\rm(\alph{blist})}}{
\usecounter{blist}\leftmargin2.5em\labelwidth2em\labelsep0.5em
\topsep0.6ex \itemsep0.3ex plus0.2ex minus0.3ex
\parsep0.3ex plus0.2ex minus0.1ex}}{\end{list}}
 
\newcounter{rlist}
\newenvironment{rlist}{\begin{list}{{\rm(\roman{rlist})}}{
\usecounter{rlist}\leftmargin2.5em\labelwidth2em\labelsep0.5em
\topsep0.6ex\itemsep0.3ex plus0.2ex minus0.3ex
\parsep0.3ex plus0.2ex minus0.1ex}}{\end{list}}

\newcommand{\xra}{\xrightarrow}

\newcommand{\Mor}{{\rm Mor}}

\newcommand{\Ra}{\Rightarrow}

\newcommand{\talpha}{{\widetilde \alpha}}
\newcommand{\tbeta}{{\widetilde \beta}}

\newcommand{\ul}{\underline}
\newcommand{\ol}{\overline}
\newcommand{\wh}{\widehat}

\newcommand{\A}{\mathbb{A}}
\newcommand{\uA}{\underrightarrow{\mathbb{A}}}
\newcommand{\uB}{\underrightarrow{\mathbb{B}}}
\newcommand{\rA}{\underline{\mathbb{A}}}

\newcommand{\rB}{\underline{\mathbb{B}}}

\newcommand{\B}{\mathbb{B}}

\newcommand{\K}{\mathbb{K}}

\newcommand{\whve}{\widehat{\varepsilon}}

\newcommand{\ot}{\otimes}

\newcommand{\ve}{\varepsilon}

\newcommand{\up}{\upsilon}

\newcommand{\id}{I}

\newcommand{\tr}{{\it{r}}}
\newcommand{\tl}{{\ell}}
\newcommand{\ttr}{{\widetilde{\it{r}}}}
\newcommand{\ttl}{{\widetilde{\ell}}}
\newcommand{\ta}{\widetilde\alpha}
\newcommand{\tb}{\widetilde\beta}
\newcommand{\twve}{\widetilde\varepsilon}  
\newcommand{\tweta}{\widetilde\eta}   

\newtheorem{theorem}{Theorem}[section]

\newtheorem{thm}[theorem]{}

\numberwithin{equation}{section}

\newcommand{\btm}{\begin{thm}}
\newcommand{\etm}{\end{thm}}

\begin{document} 
  
\title{Weak Frobenius monads and Frobenius bimodules}
 \author{Robert Wisbauer}  
% \shorttitle{Weak Frobenius monads}
% \shortauthor{R. Wisbauer}
\address{Department of Mathematics, HHU,
 40225 D\"usseldorf, Germany\newline
\phantom{ur} E-mail:  wisbauer@math.uni-duesseldorf.de,\newline
\phantom{ur} URL: www.math.uni-duesseldorf.de/~wisbauer}
 
%\thanks{The author wants to thank Bachuki Mesablishvili for proofreading}

% \research 

\date{December 8, 2015}

\begin{abstract} As observed by  Eilenberg and  Moore (1965),
for a monad $F$ with right adjoint comonad $G$ on any category $\A$,  the category
of unital $F$-modules $\A_F$ is isomorphic to the category of counital $G$-comodules $\A^G$.
The monad $F$ is Frobenius provided we have $F=G$ and then $\A_F\simeq \A^F$. 
Here we investigate which kind of isomorphisms can be obtained for non-unital monads 
and non-counital comonads. For this we observe that the mentioned isomorphism  is in
 fact an isomorphisms between $\A_F$ and the category of bimodules $\A^F_F$ 
subject to certain compatibility conditions (Frobenius bimodules). 
Eventually we obtain that 
for a weak monad $(F,m,\eta)$ and a weak comonad  $(F,\delta,\ve)$  satisfying 
$Fm\cdot \delta F = \delta \cdot m = mF\cdot F\delta$ and 
$m\cdot F\eta = F\ve\cdot \delta$, the category of compatible $F$-modules 
is isomorphic to the category of compatible Frobenius bimodules 
and the category of compatible $F$-comodules.


\end{abstract}

\maketitle  

\tableofcontents

\section*{Introduction} 

A monad $(F,m,\eta)$ on a category $\A$ is called a {\em Frobenius monad} provided
the functor $F$ is (right) adjoint to itself  (e.g. Street \cite{Street-Frob}). 
Then $F$ also allows for a 
comonad structure $(F,\delta,\ve)$ and the (Eilenberg-Moore) category $\A_F$ of 
$F$-modules is isomorphic to the category $\A^F$ of $F$-comodules. 
As shown in \cite[Theorem 3.13]{MW-QF}, this isomorphism characterises 
a functor with monad and comonad structure as Frobenius monad. 
It is not difficult to see  that the categories $\A_F$
and $\A^F$ are in fact isomorphic to the category $\A^F_F$ of 
(unital and counital) {\em Frobenius bimodules}. 
In this setting units and counits play a crucial role. 

Here we are concerned with the question what is left from these correspondences 
 when the conditions on units and counits are weakened.
 An elementary approach to this setting 
 is offered in \cite{Wi-ad-reg} and \cite{Wi-reg} where adjunctions between 
 functors are replaced by {\em regular pairings} $(L,R)$ of functors
$L:\A\to \B$, $R:\B\to \A$ (see \ref{reg-p}). 
The composition $LR$ (resp. $RL$) yields endofunctors on $\A$ (resp. $\B$) and  
these are closely related to {\em weak (co)monads} as considered 
by B\"ohm et al. in \cite{B-weak, BoLaSt-I} (see Remark \ref{rem-cat}). 
In Section \ref{Prel} we recall the definitions and collect basic results needed 
for our investigations.  

Given a non-unital monad $(F,m)$ on any category $\A$, a non-unital module
$\varrho: F(A)\to A$ is called {\em firm} (see \cite{BoGo}) if the 
defining fork
$$ \xymatrix{ 
 FF(A) \ar@<0.4ex>[rr]^{m_A}\ar@<-0.4ex>[rr]_{F(\varrho)}& & F(A) \ar[r]^{\;\varrho} & A}$$
 is a coequaliser in the category of non-unital $F$-modules.
This notion is generalised 
in Section \ref{cofirm} by restricting the coequaliser
requirement to certain classes $\K$ of morphisms of $F$-modules. It turns out that 
compatible modules of a weak monad $(F,m,\eta)$ satisfy the resulting conditions 
for a suitable class $\K$ (Proposition \ref{r-equal}). Similar results hold for 
weak comonads.

In Section \ref{Frob-Frob}, we return to parings of the functors $L$ and $R$.
Given natural transformations $\eta:I_\A\to RL$ and $\twve:RL\to I_\B$, one obtains
  a non-unital monad  $(LR,L\twve R)$ and 
 a non-counital co\-monad $(LR,L\eta R)$ on $\B$  for which the
Frobenius condition is satisfied %(see (\ref{frob-dia})) 
and this motivates
 the definition of %non-unital modules,  non-counital comodules,
Frobenius bimodules (see \ref{frob-mod}). 
Given a non-counital $LR$-comodule $\omega:B\to LR(B)$, 
the question arises when it can be 
extended to a Frobenius bimodule by some $\varrho:LR(B)\to B$. As sufficient condition 
it turns out that the defining cofork for $\varrho$ is a coequaliser in the category of 
non-counital comodules (see Proposition \ref{equal}). 
Further situations are investigated, in particular for regular pairings (Theorems \ref{reg-basic}, \ref{reg-basic-d}).

In Section \ref{weak-frob}, the results about the pairings $(L,R)$ from Section 
\ref{Frob-Frob} are reformulated for the (co)monad $LR$, that is, we consider
an endofunctor $F$ on $\B$ endowed with a weak monad structure
 $(F,m,\eta)$, a weak comonad structure $(F,\delta,\ve)$, and  
the compatibility between $m$ and  $\delta$ is postulated
as the Frobenius property (see \ref{w-frob}). (For $L\eta R$ and  $L\twve R$ 
the latter follows by naturality, see (\ref{frob-dia})). 
The constructions lead to various functors between (compatible) module, comodule and
bimodule categories (see \ref{w-firm}, \ref{w-firm-m}, \ref{push}).
For proper (co)monads we get 
some results obtained by B\"ohm and G\'omez-Torrecillas in \cite{BoGo} as 
Corollaries \ref{Cor-1}, \ref{Cor-2}.

\section{Regular pairings}\label{Prel}

Throughout $\A$ and $\B$ will denote any categories. 
%The composition of two 
%morphisms $f$ and $g$  in a category will be written as $g\cdot f$.
% or just as $gf$. 
The symbols $I_A$, $A$, or just $I$ will stand for the identity 
morphism on an object $A$,  $I_F$ or $F$ denote the identity transformation 
on the functor $F$, and $I_\A$ means the identity functor on $\A$.

Given an endofunctor $T$ on $\A$, an idempotent natural transformation $e:T\to T$
is said to {\em split} if there are an endofunctor $\ul T$ on $\A$ 
and natural transformations $p:T\to \ul T$ and $i:\ul T\to T$ such that $e= i\cdot p$ and $p\cdot i=I_{\ul T}$.

We recall some notions from \cite{Wi-ad-reg}, \cite{Wi-reg}, \cite{BoLaSt-I}.

\begin{thm}\label{q-comod}{\bf Non-counital comodules.} \em
Let $(G,\delta)$ be a pair with an endofunctor $G:\A\to \A$ and a  coassociative 
natural transformation (coproduct) $\delta:G\to GG$. Then (non-counital)
$G$-comodules are defined as objects $A\in \A$ with a morphism 
$\up:A\to G(A)$ satisfying 
$G(\up)\cdot \up=\delta_A\cdot \up$
and the category of these $G$-comodules is denoted by $\uA^G$. 

Consider a triple $(G,\delta,\ve)$, with $(G,\delta)$ a pair as above and 
$\ve:G\to I_\A$ any natural transformation (quasi-counit).
Then a $G$-comodule $(A,\up)$ is said to be {\em compatible} provided 
$\up= G\ve_A \cdot \delta_A\cdot \up$. The full subcategory of $\uA^G$ 
consisting of compatible comodules is denoted by $\rA^G$.

$(G,\delta,\ve)$ is called a {\em weak comonad} if
\begin{center}
$\ve=\ve\cdot G\ve\cdot \delta$,\; $\delta=G\ve G\cdot  G\delta\cdot \delta$, \: and \:
$G\ve\cdot \delta= \ve G\cdot \delta$. 
\end{center}
Then a $G$-comodule $(A,\up)$ is compatible if % provided 
$ \ve G_A \cdot \delta_A\cdot \up =\up= \up\cdot \ve_A \cdot \up.$ 
Furthermore, $G\ve\cdot \delta:G\to G$ is idempotent and in case this is split 
by $G\xra{p} \ul G\xra{i} G$, one obtains a comonad $(\ul G,\ul \delta,\ul \ve)$
by putting
$$\ul\delta: \ul G\xra{i} G\xra{\delta} GG\xra{pp} \ul G\ul G,\quad 
 \ul \ve:\ul G\xra{i}G\xra{\ve} I_\A.$$  
\end{thm}

\begin{thm}\label{q-mod}{\bf Non-unital modules.} \em
Let $(F,m)$ be a pair with an endofunctor 
$F:\A\to \A$ and an associative 
natural transformation (product) $m:FF\to F$. Then (non-unital) 
$F$-modules are defined as objects $A\in \A$ with a morphism 
$\varrho:F(A)\to A$ satisfying $\varrho\cdot F\varrho=\varrho\cdot m_A$
and the category of these $F$-modules is denoted by $\uA_F$. 

Consider a triple $(F,m,\eta)$, with  $(F,m)$ a pair as above and 
any natural transformation $\eta:I_\A\to F$ (quasi-unit).
An $F$-module $(A,\varrho)$ is said to be {\em compatible} provided 
$\varrho= \varrho\cdot m_A\cdot F\eta_A$ and the full subcategory of $\uA_F$ 
consisting of compatible modules is denoted by $\rA_F$.

$(F,m,\eta)$ is called a {\em weak monad} if
\begin{center}
$\eta=m\cdot F\eta\cdot \eta$,\; $m=m\cdot mF\cdot F\eta F$, \: and \:
$m\cdot F\eta= m\cdot \eta F$. 
\end{center}
Then an $F$-module $(A,\varrho)$ is compatible if 
$ \varrho\cdot m_A\cdot  \eta F_A=\varrho= \varrho \cdot \eta_A \cdot \varrho.$  
Furthermore, $m\cdot F \eta:F\to F $ is  idempotent and in case this is split 
by $F\xra{p} \ul F\xra{i} F$, one obtains a  monad $(\ul F,\ul m,\ul \eta)$ by putting
$$\ul m:\ul F\ul F \xra{ii} FF\xra{m}F\xra{p} \ul F,\quad
\ul \eta: I_\A \xra{\eta} F\xra{p} \ul F. 
$$
\end{thm}

\begin{thm}\label{pairings}{\bf Pairings of functors.} \em
For functors $L:\A\to \B$ and $R:\B\to \A$,
% between any  categories $\A$ and $\B$, 
{\em pairings} are defined as maps, natural in 
$A\in \A$ and $B\in \B$,
$$\begin{array}{c}
\xymatrix{\Mor_\B (L(A),B) \ar@<0.4ex>[r]^{\alpha} & 
 \Mor_\A (A,R(B))\ar@<0.4ex>[l]^{\beta},} \\
  \xymatrix{\Mor_\A (R(B),A) \ar@<0.4ex>[r]^{\ta} & 
 \Mor_\A (B,L(A))\ar@<0.4ex>[l]^{\tb}}.
\end{array}$$ 
These - and their compositions - are determined by natural transformations 
obtained as images of the corresponding identity morphisms,
\begin{center}
\begin{tabular}{rl|rl}
 map & natural transformation \qquad & \quad map \quad & natural transformation \\[+1mm]
 $\alpha$ & $\eta: I_\A \to RL$,& $ \ta $ & $ \tweta: I_\B\to LR$,\\ 
 $\beta$ & $\ve:LR\to I_\B$, &  $\tb $ & $\twve: RL \to I_\A$, \\
$\beta\cdot \alpha$ & $\tl:L\xra{L\eta} LRL \xra{\ve L} L$ &
          $\tb \cdot \ta  $ & $\ttr :R\xra{ R\tweta} RLR \xra{ \twve R} R$ \\
 $ \alpha\cdot \beta$ & $\tr: R\xra{\eta R} RLR \xra{R\ve} R$ &
          $\ta  \cdot \tb $ & $ \ttl: L\xra{\tweta L} LRL \xra{ L\twve} L$ .
\end{tabular}
\end{center}

$\beta$ (resp. $\alpha$) is said to be {\em symmetric} if $L\tr= \tl R$ (resp.
 $R\tl = \tr L$)  (see \cite[\S 3]{Wi-reg}). 
Under the given conditions (see \cite{Wi-reg}), 
%\begin{itemize} 
\begin{rlist}
\item $(LR, L\eta R, \ve)$ is a non-counital comonad on $\B$ with quasi-counit $\ve$;
%\item if $\beta$ is symmetric, then $(LR, L\eta R, \ve)$ is a weak comonad on $\B$;
\item $(RL,R\ve L,\eta)$ is a non-unital monad on $\A$ with quasi-unit $\eta$; 
\item $(LR, L\twve R,\tweta)$ is a non-unital monad on $\B$ with quasi-unit $\tweta$;
\item $(RL,R\tweta L,\twve)$ is a non-counital comonad on $\A$ with 
quasi-counit $\twve$. 
\end{rlist}
%\end{itemize} 
\end{thm}
 
Clearly,  if $\alpha$ is a bijection, then $(L,R)$ is an adjoint pair, if $\ta$ is a
bijection, then $(R,L)$ is an adjoint pair, and if
$\alpha$ and $\ta$ are bijections, then $LR$ and $RL$ are Frobenius functors.

\begin{thm}\label{reg-p}{\bf Regular pairings.} \em 
A pairing $(L,R,\alpha,\beta)$ is said to be {\em regular} if 
\begin{center}
$\alpha\cdot \beta\cdot \alpha =\alpha$
and $\beta\cdot \alpha\cdot \beta=\beta$.
\end{center}
In this case, $\tl: L\to L$ and $\tr:R\to R$ (see \ref{pairings}) are idempotents   
and
$$\begin{array}{c}
  \ve = \ve\cdot \tl\tr = \ve\cdot \tl R= \ve\cdot L\tr, \\[+1mm]
 \eta=\tr\tl\cdot\eta = R\tl\cdot \eta= \tr L\cdot \eta.
\end{array}$$
If $\beta$ is symmetric,  $\tl\tr= L\tr= \tl R$; if $\alpha$ is symmetric,
  $\tr\tl = R\tl = \tr L$.

Assume the idempotents $\tl$, $\tr$ to be splitting,  that is,
 $$L\xra{\tl} L = L\xra{p} \ul L\xra{i} L,\quad 
R\xra{\tr} R = R\xra{p'} \ul R\xra{i'} R.$$
Then, for the natural morphisms   
$$\ul \eta:\xymatrix{\id_\A \ar[r]^\eta &RL\ar[r]^{p' p}&\ul R \ul L,}  \quad
\ul\ve: \xymatrix{\ul L \ul R \ar[r]^{i i'}&LR\ar[r]^{\ve} & \id_\B,}
$$
one gets $\ul \ve \ul L \cdot \ul L \ul \eta = I_{\ul L}$ 
and  $\ul R \ul \ve \cdot \ul \eta \ul R= I_{\ul R}$, hence yielding 
%showing that $(\ul L,\ul R, \ul\eta, \ul\ve)$ yields 
an adjunction $(\ul L,\ul R, \ul\alpha, \ul\beta)$.
\end{thm}

\begin{thm}\label{char-reg}{\bf Proposition.} 
For functors $\xymatrix{\A \ar@<0.4ex>[r]^L & \B \ar@<0.4ex>[l]^R}$, 
there are equivalent:
\begin{blist}
\item $(L,R)$ allows for a regular pairing $(L,R,\alpha,\beta)$  
with splitting idempotents $\tl$, $\tr$;
%(notation from \ref{reg-p}, \ref{isplit});
\item there are retractions $\ul L \xra{i}  L\xra{p} \ul L$ and 
$\ul R \xra{i'}  R\xra{p'}\ul R$ such that $(\ul L,\ul R)$ allows for an adjunction.
\end{blist}
\end{thm}
\begin{proof} (a)$\Ra$(b) The data from \ref{reg-p} yield an 
adjunction $(\ul L,\ul R,\ul \alpha,\ul\beta)$ and the commutative diagram
$$\xymatrix{
\Mor_\B(L(A),B) \ar[r]^{\alpha} \ar[d]_{\Mor(i_A,B)} &
   \Mor_\A(A,R(B))  \ar[r]^{\beta} \ar[d]^{\Mor(A,p'_B)} &
   \Mor_\B(L(A),B) \ar[d]^{\Mor(i_A,B)} \\
\Mor_\B(\ul L(A),B) \ar[r]^{\ul\alpha}   &
   \Mor_\A(A,\ul R(B))  \ar[r]^{\ul\beta}   &
   \Mor_\B(\ul L(A),B)  .}
$$

(b)$\Ra$(a) Given an adjunction $(\ul L,\ul R,\ul \alpha,\ul\beta)$ and 
 retracts $\ul L\xra{i} L\xra{p}\ul L$ and  $\ul R\xra{i'} R\xra{p'}\ul R$,
the above diagram tells us how to define (new) $\alpha$ and $\beta$ to get 
 commutativity. Then it is routine to check that  $(L,R,\alpha,\beta)$ is  
a regular pairing and the resulting idempotents are split 
by $(p,i)$ and $(p',i')$, respectively.
\end{proof}

Now assume that $(L,R,\alpha,\beta)$ and $(R,L,\ta,\tb)$ are regular pairings. 
Then  $\tl$, $\ttl$ are two natural transformations on $L$ 
and $\tr$, $\ttr$ are two natural transformations on $R$.  
We are interested in the case when they coincide.
Applying \ref{char-reg} and its dual yields:
 
\begin{thm}\label{Frob-pair}{\bf Proposition.} 
For functors $\xymatrix{\A \ar@<0.4ex>[r]^L & \B \ar@<0.4ex>[l]^R}$, 
there  are equivalent:
\begin{blist}
\item $(L,R)$ allows for regular pairings $(L,R,\alpha,\beta)$  
        and $(R,L,\ta,\tb)$
with splitting idempotents $\tl=\ttl$, $\tr=\ttr$;
%(notation from \ref{reg-p}, \ref{isplit});
\item there are retractions $\ul L \xra{i}  L\xra{p} \ul L$ and 
$\ul R \xra{i'}  R\xra{p'}\ul R$ such that $(\ul L,\ul R)$ 
and $(\ul R,\ul L)$
allow  for adjunctions, that is, $(\ul L,\ul R)$ is a Frobenius pair of functors.
\end{blist}
\end{thm}

\begin{thm} \label{w-comonad} {\bf Remark.} \em
Let $(G,\delta,\ve)$ be a non-counital comonad on the category $\A$ 
with quasi-unit $\ve$. For 
the Eilenberg-Moore category $\uA\!^G$ of non-counital $G$-comodules  
there are the free and the forgetful functors
$$\phi^G:\A\to \uA\!^G,\; A\mapsto (G(A),\delta_A),  
\qquad U^G:\uA\!^G\to \A, \; (A,\omega)\mapsto A.$$

There is a pairing
$(\phi^G,U^G,\alpha^G,\beta^G)$ with the maps, for $X\in \A$, 
$(A,\omega) \in \uA\!^G$,
$$\begin{array}{rl}
\alpha^G:\Mor_\A(U^G(A),X) \to \Mor^G(A,\phi^G(X)), & f\mapsto G(f)\cdot\omega ,\\[+1mm]
\beta^G: \Mor^G(A,\phi^G(X))\to\Mor_\A(U^G(A),X), & g\mapsto \ve_X \cdot g.
\end{array}
$$ 
Compatible $G$-comodules $\up:A\to G(A)$ are those with  $\alpha^G\beta^G(\up)=\up$.

 $(G,\delta,\ve)$ is a weak comonad if and only if $(\phi^G,U^G,\alpha^G,\beta^G)$ 
is a regular pairing with $\beta^G$ symmetric (see \cite[Proposition 4.4]{Wi-reg}).
\smallskip

Similar characterisations hold  for weak monads (\cite[Proposition 3.4]{Wi-reg}).  
\end{thm}

\begin{thm}\label{symm-com}{\bf Related comonads.}  
Let $(L,R,\alpha,\beta)$  be a regular pairing (see {\rm\ref{reg-p}}). 

\begin{zlist}
\item For the coproduct
$$\delta: LR \xra{L\eta R} LRLR \xra{\tl RL\tr} LRLR,$$
 $(LR,\delta,\ve)$ is a   weak comonad  on $\B$.
If $\beta$ is symmetric,    $\delta = L\eta R$.

\item $\tl\tr:LR\to LR$ induces 
  morphisms of non-counital comonads respecting the quasi-counits, 
\begin{center}
 $(LR,L\eta R,\ve)\to (LR,L\eta R,\ve)$ and $(LR,L\eta R,\ve)\to (LR,\delta,\ve)$,   
\end{center}
and an endomorphism of weak comonads
$(LR,\delta,\ve)\to (LR,\delta,\ve)$. 
\end{zlist}
\end{thm}
\begin{proof} 
Direct verification shows 
 $\ve LR\cdot \delta =\tl\tr = LR\ve \cdot \delta$, 
the conditions for a weak comonad. 
For the next claims, consider the commutative diagram
$$\xymatrix{
 LR \ar[r]^{\tl\tr} \ar[d]_{L\eta R} & LR  \ar[d]_{L\eta R} \ar[dr]^{\delta} \\
LRLR \ar[r]^{\tl RL \tr} \ar[dr]_{\tl\tr\tl\tr} & 
     LRLR \ar[r]^{\tl RL \tr} \ar[d]^{L\tr\tl R} & 
     LRLR \ar[d]^{L\tr\tl R} \\
 & LRLR \ar[r]_{\tl RL \tr} & LRLR \,;}
$$
the left hand part proves the assertion about the first morphism and the outer paths 
show the properties of the second and third  morphisms. 
\end{proof}

\pagebreak[3]

\begin{thm}\label{symm-mon}{\bf Related monads.}  
Let $(L,R,\alpha,\beta)$  be a regular pairing (see {\rm \ref{reg-p}}). 
\begin{zlist}
\item For the product
$$m: RLRL\xra{\tr LR \tl} RLRL \xra{R\ve L} RL,$$
 $(RL,m,\eta)$ is a weak monad on $\A$.
If $\alpha$ is symmetric,  $m = R\ve L$. 
\item $\tr \tl:RL\to RL $ 
yields morphisms of non-unital monads respecting the quasi-units, 
\begin{center}
 $(RL,R\ve L,\eta)\to (RL,R\ve L,\eta)$ and $(RL,R\ve L,\eta)\to (RL,m,\eta)$.
\end{center}
and an endomorphism of weak monads   $(RL,m,\eta)\to  (RL,m,\eta)$.
\end{zlist}
\end{thm}
\begin{proof}
One easily verifies 
  $m\cdot    \eta RL= \tr\tl = m\cdot  RL  \eta $,
the condition for a weak monad.
The other claims are shown similarly to  \ref{symm-com} 
\end{proof}

Combining the preceding observations we have shown:
 
\begin{thm}\label{isplit-b}{\bf Proposition.}  
Let $(L,R, \alpha,\beta)$ be a regular pairing and assume the 
idempotents $\tl$ and $\tr$ to split. With the notation from {\rm\ref{reg-p}}, 
$(\ul L \ul R, \ul L \ul\eta \ul R, \ul\ve)$ is a comonad on $\B$ and 
$(\ul R\ul L, \ul R \ul\ve \ul L, \ul\eta)$ is monad on $\A$. Then, 
\begin{zlist}
\item  the natural transformation $pp':LR \to \ul L\ul R$ induces 
  morphisms of non-counital comonads 
$(LR,L\eta R,\ve) \to (\ul L \ul R, \ul L \ul\eta \ul R, \ul\ve) $, and 
morphisms of weak comonads 
 $(LR,\delta,\ve)\to (\ul L \ul R, \ul L \ul\eta \ul R, \ul\ve)$; 
\item
 the natural transformation $p'p:RL \to \ul R\ul L$ induces 
 morphisms  of non-unital monads
$(RL,R\ve L,\eta)\to (\ul R\ul L, \ul R \ul\ve \ul L, \ul\eta)$\; and \;
 morphisms  of weak monads
 $(RL,m,\eta) \to(\ul R\ul L, \ul R \ul\ve \ul L, \ul\eta)$.
\end{zlist}
\end{thm}

\begin{thm}\label{comp-mod-reg}{\bf Regular pairings and comodules.} \em
Let $(L,R,\alpha, \beta)$ be a regular pairing and  consider the 
weak comonad $(LR,\delta,\ve)$ defined in \ref{symm-com}.
Then a non-counital $(LR,\delta,\ve)$-comodule $(B,\up)$ is 
compatible  (see  \ref{q-comod}) if  
$\up = \ve LR_B \cdot \delta_B \cdot \up = \tr\tl_B \cdot \up$.
 
Write $\ol\B^{LR,\delta}$ for the full subcategory
 of $\uB^{LR,\delta}$ formed by the compatible $(LR,\delta,\ve)$-co\-modules.
For any $B\in \B$,  $(LR(B),\delta_B)$ is a 
compatible $(LR,\delta,\ve)$-comodule, and thus we have a functor
$$\phi^{LR,\delta}:\B \to \ol\B^{LR,\delta}, \quad B\mapsto (LR(B), \delta_B).$$
The obvious forgetful functor $U^{LR,\delta}:\ol\B^{LR,\delta}\to \B$ need not be
 (left) adjoint to $\phi^{LR}$ but $(\phi^{LR},U^{LR,\delta})$ allows for 
a regular pairing (see \ref{w-comonad}).  %\cite[4.2]{Wi-reg}).   

Denoting by $\uB^{LR,\eta}$ the non-counital comodules for $(LR,L \eta R, \ve)$, 
the natural transformation  $(LR,L \eta R, \ve)\to (LR,\delta, \ve)$
induced by $\tl\tr$ (see \ref{symm-com}) defines a functor 
$t_{\tl\tr}:\uB^{LR,\eta }\to\uB^{LR,\delta}$. It is easy to see that hereby the image of 
any comodule in $\uB^{LR,\eta }$ is a compatible comodule in $\uB^{LR,\delta}$
leading to a commutative diagram
$$\xymatrix{ \B \ar[rr]^{\phi^{LR,\eta}}  \ar[drr]_{\phi^{LR,\delta}}  &&
    \uB^{LR,\eta } \ar[d]^{t_{\tl\tr}\quad} \\
 &&\ol\B^{LR,\delta} . } $$

In case the idempotents $\tl$ and $\tr$ are splitting,
we get the splitting natural transformation  $pp':LR\to \ul L \ul R$ (from \ref{reg-p}) 
which induces functors $\ol\B^{LR,\delta} \to \ol\B^{\ul L\ul R}$ and 
$\ol\B^{LR,\eta} \to \ol\B^{\ul L\ul R}$, also denoted by $pp'$, 
with commutative diagram
$$\xymatrix{  
    \uB^{LR ,\eta  } \ar[drr]^{pp'}  \ar[rr]^{t_{\tl\tr}} &&\ol\B^{LR,\delta}  \ar[d]^{pp'}  \\
\B \ar[u]^{\phi^{LR,\eta}} \ar[rr]_{\phi^{\ul L\ul R}}
 &&\ol\B^{\ul L\ul R}   . } $$

Since $\ul L\ul R$ is a comonad, every non-counital $\ul L\ul R$-comodule is compatible, 
that is $\ol\B^{\ul L\ul R} = \uB^{\ul L\ul R} $, but need not be counital. 
\end{thm}

\begin{thm}\label{rem-cat}{\bf Remark.} \em 
As pointed out by an anonymous referee,
a regular pairing $(L,R,\alpha,\beta)$ defined in \ref{reg-p} is in fact the same as an
adjunction in the local idempotent closure $\overline{{\rm Cat}}$ of the 2-category
${\rm Cat}$ of 
categories and hence corresponds to a comonad in $\overline{\rm Cat}$.
This lives on the 1-cell $(LR,\tl\tr)$ with coproduct $\tl RL\tr\cdot  L\eta R$
and counit $\ve$ (see \cite{BoLaSt-I}).
In this approach, similar to Proposition \ref{Frob-pair}, the properties
of the weak comonad $LR$ are described by properties of a related comonad $\ul L\ul R$.

We are also interested in the modules and comodules induced directly  by $RL$ and
$LR$, respectively.
\end{thm}

\section{(Co)firm (co)modules}\label{cofirm}

To develop further constructions for pairings of functors,  
 symmetry conditions are needed and so we consider weak (co)monads.   

The notion of (co-)equalisers in categories may be modified in the following way.

\begin{thm}\label{def-equ}{\bf Definitions.} \em 
Let $\K$ be a class of morphisms in a category $\A$ closed under composition. A cofork
$$ \xymatrix{ 
B\ar[r]^{k} & C \ar@<0.4ex>[r]^{g}\ar@<-0.4ex>[r]_{f}
& D }$$
is said to be a   {\em $\K$-equaliser}  provided $k\in \K$ %, $f$ and $g$ 
 and, for any 
$h:Q\to C$ in $\K$ with $ f\cdot h =g\cdot h$,  
there exists a unique $q:Q\to B$ in $\K$ such that $h=k\cdot q$.  
If this holds, then, for morphisms $r,s:X\to B$ in $\K$, $k\cdot r=k\cdot s$ implies $r=s$.

Similarly,  a fork
$$ \xymatrix{ 
 B \ar@<0.4ex>[r]^{g}\ar@<-0.4ex>[r]_{f} & C \ar[r]^{s} & D}$$
is said to be a {\em $\K$-coequaliser} provided $s\in \K$ %, $f$ and $g$ 
 and, for any 
$h: C\to Q$ in $\K$ with $ h\cdot f =h\cdot g$,  
there is a unique $q:D\to Q$ in $\K$ such that $h=q\cdot s$.  
In this case, for morphisms $t,u: D\to Y$ in $\K$,  $t\cdot s= u\cdot s$ implies $t=u$.

A class $\K$ of morphisms in $\A$ is called an {\em ideal class} if for any
morphisms   $A\xra{f} B \xra{g} C$ in $\A$,  $f$ or $g$  in  
$\K$ implies that $g\cdot f$ is in $\K$.

Taking for $\K$ the class of all morphisms in $\A$, the  notions defined above yield 
the usual equalisers and coequalisers in the category $\A$.  
\end{thm}  

\begin{thm}\label{comp-mod}{\bf $\K$-cofirm comodules.} \em
 Let $(G,\delta)$ be a non-counital comonad. 
% with quasi-unit $\ve$.  
Given an ideal class $\K$ of morphisms in
the category   $\uB^{LR}$  of non-counital $G$-comodules, 
a comodule $(B,\omega)$ is called  {\em $\K$-cofirm}
provided the defining cofork 
$$\xymatrix{
 B\ar[r]^{\omega\quad} & 
 G(B)\ar@<0.4ex>[rr]^{\delta_B\quad}\ar@<-0.4ex>[rr]_{\, G(\omega)\quad}&& 
 GG(B)}$$
is a $\K$-equaliser. If we choose for $\K$ all morphisms in $\uB^{LR}$, 
a $\K$-cofirm comodule is just called {\em cofirm}.
\end{thm}

\begin{thm}\label{comp-morph}{\bf Compatible comodule morphisms.} \em
Now let $(G,\delta,\ve)$ be a weak comonad and $\gamma:= G\ve \cdot\delta:G\to G$ 
the idempotent comonad morphism.
 We call a morphism $h$ between $G$-comodules $(B,\omega)$ and $(B',\omega')$ 
{\em $\gamma$-compatible}, provided it induces commutativity of the triangles in the diagram
$$\xymatrix{ B \ar[d]_h \ar[r]^{\omega\quad} \ar[drr]^h & G (B) \ar[r]^{\quad\ve_B} & B \ar[d]^h\\
 B' \ar[r]_{\omega'\quad} &  G (B') \ar[r]_{\quad\ve_{B'} } & B' .} $$
Clearly, since the outer diagram is always commutative for comodule morphisms, 
it is enough to require commutativity for one of the triangles.
Thus one readily obtains:
{\em
\begin{zlist}
\item The class $\K_{\gamma}$  of all $\gamma$-compatible morphisms in 
    $\ol\B^{G}$ is  an ideal class.
\item A morphism $h:Q\to G (B)$ of $G $-comodules is in $\K_{\gamma}$
  if and only if \; $\gamma_B \cdot h = h$. 
\item A morphism $h:G(B)\to Q$ of $G$-comodules is in $\K_{\gamma}$  if and only if 
 $h\cdot {\gamma}_B= h$.
\end{zlist}}
\smallskip 
     
Evidently, a  $G$-comodule $(B,\omega)$ is compatible (as in \ref{q-comod})
if and only if \; $\omega \in \K_{\gamma} $, that is, $\omega={\gamma}_B\cdot \omega$.

Notice that ${\gamma}=I_{G}$ implies that every non-counital $G $-comodule 
is ${\gamma}$-compatible, that is, $\uB^{G}=\ol\B^{G}$; in this case, however, 
not every $G $-comodule morphism  
need to be ${\gamma}$-compatible and a $G$-comodule $(B,\omega)$ need not be counital 
but only satisfies $\omega=\omega\cdot\ve_B\cdot \omega$.
\end{thm}
 
\begin{thm}\label{d-equal}{\bf Proposition.}
If  $(G,\delta,\ve)$ is a weak comonad,
then any compatible $G$-comodule $(B,\omega)$ is  $\K_{\gamma}$-cofirm.
\end{thm}
\begin{proof} 
  We have to show that the cofork
$$\xymatrix{
 B\ar[r]^{\omega\quad} & 
 G (B)\ar@<0.4ex>[rr]^{\delta_B \quad}\ar@<-0.4ex>[rr]_{\, G (\omega)\quad}&& G G (B)}$$
is a  $\K_{\gamma}$-equaliser.
Let $(Q,\kappa)$ be a  $G $-comodule and  
  $h: Q \to G (B)$ a morphism in $\K_{\gamma}$ 
with $G (\omega) \cdot h = \delta_B\cdot h$.  
In the diagram
\begin{equation}\label{dia-equal}
\xymatrix{
 Q\ar[rr]^{h\quad}\ar[dr]^{h\quad}\ar[dd]_{\kappa} & 
 & G (B) \ar[d]^{G (\omega)} \ar[r]^{\quad\ve_B}
     & B \ar[dd]^\omega \\
 & G (B) \ar[r]^{\delta_B} \ar[d]^{\delta_B} & G G (B) \ar[dr]^{\ve_{ G(B)}}   \\
G (Q) \ar[r]_{G (h)\quad}  &
 G G (B) \ar[rr]_{\;G \,\ve_B} & & G (B) ,
}
\end{equation} 
 all inner diagrams are commutative.
This shows that $\widetilde h := \ve_B\cdot h:Q \to B$ is a $G$-comodule morphism with 
$$\begin{array}{rcl}
 \omega \cdot\widetilde h  & =& \omega\cdot\ve_B\cdot h 
 \;= \; \ve_{G (B)} \cdot G (\omega) \cdot h \\
& = & \ve_{G (B)}  \cdot \delta_B\cdot h  
\;=\; \gamma_B\cdot h\; =\; h,\\[+1mm]
\ve_B\cdot \omega \cdot\widetilde h &=& \ve_B \cdot \gamma_B \cdot h  
\; = \; \ve_B \cdot h \; =\; \widetilde h,
\end{array}$$ 
thus $\widetilde h \in \K_{\gamma}$.
Moreover, for any 
$q: Q\to B$ in $\K_\gamma$ with $\omega \cdot q= h$,  we have
  $\ve_B \cdot h = \ve_B \cdot \omega \cdot q = q$,
showing uniqueness of $q$.
\end{proof}

Replacing $(Q,h)$ in diagram  (\ref{dia-equal}) by $(B,\omega)$, we see that 
$\ve_B\cdot \omega$ is a comodule morphism and this
leads to the following observation.

\begin{thm}\label{equ-counit}{\bf Proposition.} 
 If  $(G,\delta,\ve)$ is a (proper) comonad,
then any  non-counital  
 $G$-comodule $(B,\omega)$
 is cofirm if and only if it is counital. 
\end{thm}
\begin{proof}  
Since we have a comonad, $\gamma =I_{G}$, every $G$-comodule $(B,\omega)$
 is $\gamma$-compatible, and $\omega=\omega\cdot \ve_B\cdot\omega$
(see \ref{comp-morph}). 

If $(B,\omega)$ is cofirm, then $\omega$ is monomorph in $\uB^{LR}$;  since 
$\ve_B\cdot\omega$ and $I_{B}$ are morphisms in $\uB^{G}$ we conclude
$\ve_B\cdot\omega=I_{B}$, that is, $(B,\omega)$ is counital.

 It is folklore that any counital $G$-comodule is cofirm. 
\end{proof}
 
\begin{thm}\label{monads}{\bf $\K$-firm modules.} \em
Let $(F,m)$ be a non-unital monad  on $\B$.
Given an ideal class of morphisms in the category $\uB_{F}$ 
of non-unital $F$-modules (see \cite{Wi-reg}),
a module $(B,\varrho)$ is called  {\em $\K$-firm} provided the defining fork 
$$ \xymatrix{ 
 FF(B) \ar@<0.4ex>[rr]^{\quad m_B}\ar@<-0.4ex>[rr]_{\quad F(\varrho)}
& & F(B) \ar[r]^{\quad\varrho} & B }$$
is a $\K$-coequaliser (Definitions \ref{def-equ}). 
\end{thm}

\begin{thm}\label{rem-monads}{\bf Remark.} \em
Following \cite[2.3]{BoGo}, a non-unital  $F$-module $(B,\varrho)$ is called {\em firm}
provided it is $\K$-firm for the class $\K$ of all morphisms in $\uB_{F}$
 and $\varrho$ is an epimorphism in $\B$. The term {\em firm} was coined by 
Quillen for non-unital algebras $A$ over a commutative ring $k$ with the property
that the map $A\ot_A A \to A$, $a\ot b\mapsto ab$, is an isomorphism. 
Then, an $A$-module is firm provided it is firm for the monad $A\ot_k-$ 
on the category of $k$-modules. In the category of non-unital 
$A$-modules, coequalisers are induced by coequalisers of $k$-modules and hence are   
epimorph (in fact surjective) as $k$-module morphisms (e.g. \cite[6.1]{BoGo}). 
\end{thm} 

\begin{thm}\label{comp-morph-mod}{\bf Compatible module morphisms.} \em
Let  $(F,m,\eta)$ be a weak monad with idempotent monad morphism 
$\vartheta:=m\cdot \eta F:F\to F$. 
A morphism $h$ between $F$-modules $(B,\varrho)$ and $(B',\varrho')$ is called 
{\em $\vartheta$-compatible}, provided it induces commutativity of the triangles in the
diagram
$$\xymatrix{ 
B \ar[d]_h \ar[r]^{\eta_B\quad} \ar[drr]^h & F(B) \ar[r]^{\quad\varrho} & B \ar[d]^h\\
B' \ar[r]_{\eta_{B'}\quad} &  F(B') \ar[r]_{\quad\varrho'} & B' .} $$
 
Similar to \ref{comp-morph} one obtains:
{\em
\begin{zlist}
\item 
The class $\K_\vartheta$  of all $\vartheta$-compatible morphisms in $\rB_{F}$ is an 
ideal class.
\item A morphism $h:Q\to F(B)$ of $F$-modules is in $\K_\vartheta$
  if and only if   $\vartheta_B \cdot h = h$. 
\item A morphism $h:LR(B)\to Q$ of $F$-modules is in $\K_\vartheta$  if and only if  
  $h\cdot \vartheta_B = h$.
\end{zlist}}
\smallskip

Clearly, an $F$-module $(B,\varrho)$ is compatible (see \ref{q-mod})
if and only if  $\varrho \in \K_\vartheta$, that is, 
$\varrho\cdot \vartheta_B = \varrho$.
\end{thm}

\begin{thm}\label{rem-comp}{\bf Remark.} \em
Given the assumptions in \ref{comp-morph-mod}, one may consider the subcategory of 
$\rB_{F}$ consisting of the same objects and as morphisms the $\vartheta$-compatible 
morphisms.
Then the identity morphism on a $\vartheta$-compatible module $(B,\varrho)$
is $\varrho\cdot\eta_B: B\to B$ and  equalisers in this category are essentially the $\K_\vartheta$-equalisers. 
This situation is also addressed in \cite[Remark 2.5]{BoLaSt-I}  (with different terminology). 
\end{thm}

Dual to the Propositions \ref{d-equal} and \ref{equ-counit} we now have:

\begin{thm}\label{r-equal}{\bf Proposition.}
If  $(F,m,\eta)$ is a weak monad,
then any $\vartheta$-compatible $F$-module $(B,\varrho)$ is $\K_\vartheta$-firm.
\end{thm}

\begin{thm}\label{equ-unit}{\bf Proposition.} 
If  $(F,m,\eta)$ is a (proper) monad,
then a non-unital $F$-module $(B,\varrho)$ is firm if and only if it is unital.
\end{thm}

\section{Frobenius property and Frobenius bimodules}\label{Frob-Frob}

In the setting of \ref{pairings}, assume $\alpha$ and $\tbeta$ to be given, that is,
there are natural transformations  $\eta:I_\A\to RL$ and $\twve:RL\to I_\A$.
Then $(LR,L\eta R)$ is a non-counital comonad, 
and $(LR,L\twve R)$ is a non-unital monad on $\B$ 
(see \cite{Wi-reg}). 
This section is for studying the interplay between the corresponding module and
 comodule structures.

 Let $\uB^{LR}_{LR}$ denote the category of objects in $\B$ which have an 
$LR$-module as well as an $LR$-comodule structure ($LR$-bimodules)
 and with morphisms which are $LR$-module and $LR$-comodule morphisms. 

By naturality, we have the commutative diagram (Frobenius property)
\begin{equation}\label{frob-dia}
\xymatrix{ & LRLRLR  \ar[rd]^{LRL\twve R} \\
LRLR  \ar[ru]^{ L\eta RLR} \ar[rd]_{LRL\eta R}\ar[r]^{\quad L\twve R } & 
LR \ar[r]^{L\eta R \quad} & LRLR \\
     & LRLRLR \ar[ru]_{L\twve RLR} & .} 
\end{equation}
 
We are interested in $LR$-modules and $LR$-comodules subject to a reasonable 
compatibility condition.

\begin{thm}\label{frob-mod}{\bf Frobenius bimodules.} \em
A triple $(B,\varrho,\omega)$ with an object $B\in \B$ and 
two morphisms $\varrho:LR(B)\to B$ and $\omega:B\to LR(B)$
is called a {\em Frobenius bimodule} provided the data induce commutativity of the diagram  
\begin{equation*}
\xymatrix{
LRLR(B) \ar[r]^{\quad LR(\varrho)} \ar[d]_{L\twve R}\ar@{}[rd]|{{\rm (I)}} &
 LR(B) \ar[r]^{LR(\omega)\quad} \ar[d]^{\varrho} \ar@{}[rd]|{{\rm (II)}}&
  LRLR(B) \ar[d]^{L\twve R} \\
LR(B)\ar[r]^{\quad \varrho} \ar[d]_{L\eta R} \ar@{}[rd]|{{\rm(III)}} &
  B \ar[r]^{\omega \quad} \ar[d]^\omega \ar@{}[rd]|{{\rm (IV)}} & LR(B) \ar[d]^{L\eta R} \\
LRLR(B) \ar[r]_{\quad LR(\varrho)} & LR(B) \ar[r]_{LR(\omega)\quad} & LRLR(B) .}
\end{equation*}
This implies that  $\varrho:LR(B)\to B$ defines a (non-unital) $LR$-module  
and  
$\omega:B\to LR(B)$ a (non-counital) $LR$-comodule;  if that is already known,
the conditions on Frobenius bimodules reduce to commutativity of the diagrams
(II) and (III), that is commutativity of (Frobenius property for modules)
\begin{equation}\label{frob-prop}
\xymatrix{  & LRLR(B) \ar[dr]^{LR(\varrho)}   \\
LR(B) \ar[r]^{\quad\varrho} \ar[rd]_{LR(\omega)} \ar[ru]^{L\eta R} & B
 \ar[r]^{\omega\quad} & LR(B) \\
  & LRLR(B) \ar[ru]_{L\twve R} & . 
}
\end{equation}

Denote by $\B^{LR}_{LR}$ the 
category with the Frobenius $LR$-bimodules as objects and mor\-phisms which are 
$LR$-module as well as $LR$-comodule morphisms.
 
By the commutative diagram  (\ref{frob-dia}), for any $B\in \B$,
$LR(B)$ is a Frobenius bimodule with the canonical structures, that is, there is a 
%comparison
 functor 
$$K^{LR}_{LR}:\B \to \B^{LR}_{LR},\quad B \mapsto (LR(B), L\twve_{R(B)}, L\eta_{R(B)}) . $$
\end{thm}

\begin{thm}\label{sep}{\bf Natural mappings.} \em
Assume again $\eta:I_\A\to RL$ and $\twve:RL\to I_\A$ to be given (see \ref{pairings}). 
Then there are maps, natural in $A,A'\in \A$,
$$ \begin{array}{rl}
\Phi_{A,A'}: \Mor_\B(L(A),L(A'))\to \Mor_\A(A,A'), & 
   g\mapsto \twve_{A'}\cdot R(g)\cdot\eta_A,\\[+1mm]
 L_{A,A'}:  \Mor_\A(A,A')\to \Mor_\B(L(A),L(A')), &  
   f \mapsto L(f),  
\end{array} $$   
$$\Phi_{A,A'}\cdot L_{A,A'}: \Mor_\A(A,A')\to \Mor_\A(A,A'),\;
 f\mapsto f\cdot \twve_A\cdot\eta_A = \twve_{A'}\cdot\eta_{A'}\cdot f.
  $$
\begin{itemize}
\item If $\twve\cdot \eta=I_\A$, then $\Phi\cdot L_{-,-}$  is the identity  
       ($L$ is  separable). 
\item If $\eta \cdot \twve\cdot \eta =\eta$, then 
$\Phi\cdot L_{-,-}\cdot\Phi = \Phi$
($\Phi\cdot L_{-,-}$ is idempotent).
\end{itemize}

The natural transformation
  $$\theta: LR\xra{L\eta R} LRLR \xra{L\twve R} LR$$
  is an $LR$-module as well  as an $LR$-comodule morphism. 
From diagram (\ref{frob-dia}) one immediately obtains the equalities
$$\begin{array}{c}
L\eta R\cdot \theta = LR\, \theta \cdot L\eta R = \theta\, LR \cdot L\eta R, \\[+1mm]
\theta \cdot L\twve R = L\twve R\cdot \theta\, LR = L\twve L\cdot LR\,\theta.
\end{array}$$
\end{thm}

Similar relations are obtained for Frobenius bimodules.

\begin{thm}\label{sep-mod}{\bf Proposition.}
Given $\eta:I_\A\to RL$ and $\twve:RL\to I_\A$, let $(B,\varrho,\omega)$ be a Frobenius  $LR$-bimodule (see \ref{frob-mod}). Then 
\begin{center}
$ \varrho\cdot \omega \cdot \varrho = \varrho\cdot \theta_B$ \; and \;
  $\omega \cdot \varrho\cdot \omega =\theta_B\cdot \omega.$
\end{center}
\begin{zlist}
\item If  $\twve\cdot \eta=I_\A$, then  $\varrho\cdot \omega \cdot \varrho = \varrho$ and
  $\omega \cdot \varrho\cdot \omega = \omega.$ \\
Then, if $\varrho$ is an epimorphism in $\uB^{LR}$ or $\omega$ is a monomorphism in $\uB_{LR}$,
one gets $\varrho\cdot\omega =I_B$. 

\item If  $\eta\cdot\twve\cdot \eta =\eta$   
or $ \twve\cdot \eta \cdot \twve =\twve$, 
then  $\omega \cdot \varrho$ is an idempotent morphism.
\end{zlist}
\end{thm}
\begin{proof}  The equalities claimed and (1) can be derived from the commutative diagram
$$
\xymatrix{& & LRLR(B)\ar[dr]^{LR(\varrho)}\ar[r]^{\quad L\twve_{R(B)}} & 
  LR(B)\ar[dr]^\varrho \\
B\ar[r]^{\omega\quad} \ar[dr]_\omega &
   LR(B) \ar[r]^{\quad\varrho} \ar[rd]^{LR(\omega)} \ar[ru]^{L\eta_{R(B)}} & B
 \ar[r]^{\omega\quad} & LR(B) \ar[r]_\varrho & B \\
& LR(B)\ar[r]_{L\eta_{R(B)}\quad}  & LRLR(B) \ar[ru]_{L\twve_{R(B)}} & & . 
}
$$

(2) To show this, extend the above diagram by $\omega$ on the right or by $\varrho$ on the left, respectively.
\end{proof}

\begin{thm}\label{comp-morph-bi}{\bf Compatible bimodule morphisms.} \em
Assume  $\eta:I_\A\to RL$ and $\twve:RL\to I_\A$ to be given. 
A morphism $h$ between Frobenius modules $(B,\varrho,\omega)$ and 
$(B',\varrho',\omega')$ is called 
{\em $\theta$-compatible}, provided it induces commutativity of the diagram
$$\xymatrix{ 
B \ar[d]_h \ar[r]^{\omega\quad} \ar[drr]^h & LR(B) \ar[r]^{\quad\varrho} & B \ar[d]^h\\
B' \ar[r]_{\omega'\quad} &  LR(B') \ar[r]_{\quad\varrho'} & B' .} $$
One easily obtains the following.
{\em
\begin{zlist}
\item 
The class $\K_\theta$  of all $\theta$-compatible bimodule morphisms in $\uB_{LR}^{LR}$ 
is an ideal class.
\item A morphism $h:Q\to LR(B)$ of $LR$-bimodules is in $\K_\theta$
  if and only if \; $\theta_B \cdot h = h$. 
\item A morphism $h:LR(B)\to Q$ of $LR$-bimodules is in $\K_\theta$  if and only if \;
  $h\cdot \theta_B = h$.
\item If $\twve\cdot\eta\cdot\twve=\twve$, then $L\twve R=\theta\cdot L\twve R$,
that is, $L\twve R$ is $\theta$-compatible.

\item If $\eta\cdot\twve\cdot \eta=\eta$, then $L\eta R=L\eta R \cdot\theta$,
that is, $L\eta R$ is $\theta$-compatible.

\item For a Frobenius bimodule $(B,\omega,\varrho)$,
$\omega$ is $\theta$-compatible if and only if $\omega=\omega\cdot\varrho\cdot \omega$,
and $\varrho$ is $\theta$-compatible if and only if 
$\varrho=\varrho\cdot \omega\cdot\varrho$.
\end{zlist}
}
\end{thm}

The next result shows how (co)firm (co)modules enter the picture.

\begin{thm}\label{prop-w-sep}{\bf Proposition.}
Let  $\eta:I_\A\to RL$ and $\twve:RL\to I_\A$ be given
and consider a Frobenius $LR$-bimodule $(B,\varrho,\omega)$.
\begin{zlist}
\item 
  If $\omega$ is $\theta$-compatible, then $(B,\omega)$ is $\K_\theta$-cofirm;\\
 if $\varrho$ is $\theta$-compatible, then $(B,\varrho)$ is $\K_\theta$-firm.
 
\item If $\twve\cdot \eta\cdot\twve =\twve$, then $(LR(B),L\twve_{R(B)})$ is a $\K_\theta$-firm module;  \\
   if $\eta\cdot \twve\cdot \eta  =\eta$,
$(LR(B),L\eta_{R(B)})$ is a $\K_\theta$-cofirm comodule.
\end{zlist}
\end{thm}
\begin{proof} (compare Proposition \ref{d-equal})  (1)
For a  non-counital $LR$-comodule $(Q,\kappa)$, let $h:Q\to LR(B)$ be a  
comodule morphism with   $L\eta R \cdot h =LR(\omega)\cdot h$ and
$h=\theta\cdot h$. For $\widetilde h:=\varrho \cdot h$ 
we get 
$$ \omega\cdot \widetilde h= \omega \cdot\varrho \cdot h
  =L\twve R\cdot LR(\omega)\cdot h = L\twve R\cdot L\eta R\cdot h = h. $$
For any  $\theta$-compatible comodule morphism $q:Q\to B$ with $\omega \cdot q=h$,
we have  $g=\varrho \cdot \omega\cdot q =\varrho \cdot h =\widetilde h$, showing uniqueness of $\widetilde h$.

The second claim is shown similarly.
\smallskip

(2) In view of \ref{comp-morph-bi}, (4) and (5), the assertions follow from (1).
\end{proof}

\begin{thm}\label{equal}{\bf Proposition.}
Assume  $\eta:I_\A\to RL$ and  $\twve:RL\to I_\A$ to be given.
Let $\K$ be  an ideal class of $LR$-comodule morphisms and suppose 
$L\twve_{R(B)}$ in $\K$ for any $B\in \B$.
\begin{zlist}
\item If $(B,\omega)$ in $\uB^{LR}$ is a $\K$-cofirm comodule (see {\rm \ref{comp-mod}}),                                                                                                                               
 there is a unique $\varrho: LR(B)\to B$ in $\K$ making
$(B,\varrho,\omega)$ a Frobenius bimodule.
 
\item With this module structure, $LR$-comodule morphisms between $\K$-cofirm
$LR$-comodules  $(B,\omega)$  and $(B',\omega')$ are  
morphisms of the Frobenius bimodules $(B,\omega,\varrho)$ and  
$(B',\omega',\varrho')$. 
\end{zlist}
\end{thm}
\begin{proof} (1) Consider the diagram (see \ref{frob-mod})
$$\xymatrix{
LRLR(B)\ar[d]_{L\twve R} \ar@{.>}[r]^{LR(\varrho)} \ar@{}[dr]|{{\rm (I)}} & 
    LR(B) \ar@{.>}[d]^\varrho \ar[r]^{L R(\omega)\quad} \ar@{}[dr]|{{\rm (II)}} & LRLR(B) \ar[d]^{L\twve R}  \\
LR(B) \ar@{.>}[r]^{\quad\varrho} \ar[d]_{L\eta R} \ar@{}[dr]|{{\rm (III)}} & 
B \ar[r]^{\omega\quad} \ar[d]^\omega  \ar@{}[dr]|{{\rm (IV)}} & LR(B) \ar[d]^{L\eta R} \\
  LRLR(B) \ar@{.>}[r]_{LR(\varrho)}  &  LR(B) \ar[r]_{LR(\omega)\quad} & LRLR(B) ,
}  
$$
where (IV) is assumed to be a  $\K$-equaliser. Since 
$$ \begin{array}{rcl}
 L\eta_{R(B)}\cdot L\twve_{R(B)}\cdot LR(\omega) &=
  & L\twve_{RLR(B)}\cdot L RL\eta_{R(B)}  \cdot LR (\omega)  \\
   &=& L \twve_{RLR(B)} \cdot  LRLR(\omega)     \cdot LR(\omega)   \\ 
 &=& L R (\omega)  \cdot  L\twve_{R(B)}  \cdot LR(\omega) ,  
\end{array} $$
and $ L\twve R_B  \cdot LR(\omega)$  
is in $\K$,
there exists a unique $\varrho:LR(B)\to B$ in $\K$ leading to the commutative 
diagram (II), and  (III) commutes since $\varrho$ is required to be a comodule morphism.  Moreover, 
$$ \begin{array}{rcl}
\omega \cdot \varrho \cdot L\twve_{R(B)} &=
  & LR(\varrho) \cdot L\eta_{R(B)} \cdot L\twve_{R(B)} \\
&=& LR(\varrho) \cdot L\twve_{RLR(B)}  \cdot LRL\eta_{R(B)} \\
&=&   L\twve_{R(B)} \cdot LRLR(\varrho)  \cdot LRL \eta_{R(B)} \\
&=&   L\twve_{R(B)} \cdot LR(\omega)  \cdot LR(\varrho)   
    \;=\; \omega \cdot \varrho \cdot LR(\varrho) ,  
\end{array} $$
and hence $\varrho \cdot L\twve_{R(B)} =\varrho \cdot LR\varrho$ since 
$\omega$ is a $\K$-equaliser. This means that the diagram  (I) is also commutative.
\smallskip

(2) Now let $h:B\to B'$ be an $LR$-comodule morphism. Then 
$$\begin{array}{rcl}
 \omega' \cdot h\cdot \varrho &=& LR(h) \cdot \omega \cdot \varrho\\
 &=& LR(h) \cdot L\twve_{R(B)} \cdot LR(\omega) \\
 &=& L\twve_{R(B')} \cdot LRLR(h) \cdot LR(\omega) \\
&=& L\twve_{R(B')}\cdot LR(\omega') \cdot LR(h)   
\; = \; \omega' \cdot \varrho' \cdot LR(h)  
\end{array}
$$
and, since both $h\cdot \varrho$ and  $\varrho' \cdot LR(h)$  are in $\K$, 
this implies that they are equal (see Definition \ref{def-equ}), that is, $h$ is also an 
$LR$-module morphism.
\end{proof}

Symmetric to Proposition \ref{equal} we get:

\begin{thm}\label{equal-d}{\bf Proposition.}
Assume  $\eta:I_\A\to RL$ and  $\twve:RL\to I_\A$ to be given.
Let $\K'$ be  an ideal class of $LR$-module morphisms
and suppose $L\eta_{R(B)}$ belongs to $\K'$ for any $B\in \B$.
\begin{zlist}
\item If $(B,\varrho)$ in $\uB_{LR}$ is a $\K'$-firm module (see {\rm\ref{monads}}),                                                                                                                             
 there is a unique $\omega: B\to LR(B)$ in $\K'$ making
$(B,\varrho,\omega)$ a Frobenius bimodule.
 
\item With this comodule structure, $LR$-module morphisms between $\K'$-firm
$LR$-modules  $(B,\varrho)$  and $(B',\varrho')$ are  
morphisms of the Frobenius bimodules $(B,\omega,\varrho)$ and  
$(B',\omega',\varrho')$. 
\end{zlist}
\end{thm}
 
So far we have only considered the case when $\alpha$ and 
$\tb$ (in \ref{pairings}) exist. Now we want to include more 
mappings in our assumptions.

\begin{thm}\label{theta-gamma}{\bf Lemma.}
Refer to the notation in {\rm \ref{pairings}} and {\rm \ref{sep}}.
\begin{zlist}
\item  Let  $(L,R,\alpha, \beta)$ be any pairing  and 
 $\twve: RL\to I_\A$ a natural transformation satisfying  
 $\eta\cdot\twve\cdot \eta =\eta$.
Then \; $\tl \tr\cdot \theta  =\tl\tr$.
\item 
 Let  $(R,L,\talpha, \tbeta)$ be any pairing  and $\eta:I_\A\to RL$ a
natural transformation satisfying  $\twve\cdot \eta\cdot\twve  =\twve$.
Then \; $ \theta\cdot\ttl \ttr=\ttl\ttr$.
\end{zlist} 
\end{thm}
\begin{proof}
The assertions  follow immediately from the definitions.
\end{proof}

\begin{thm}\label{reg-basic}{\bf Theorem.}
Let $(L,R,\alpha, \beta)$ be a regular pairing with $\beta$ symmetric
and $\twve:RL\to I_\A$ any natural transformation.  
Then,  
\begin{zlist}
\item $\whve := \twve\cdot \tr\tl :  RL\to I_\A$ is a natural 
transformation with $\whve = \whve \cdot \tr\tl$. Furthermore, 
$\tl\tr \cdot L\whve R = L\whve R$, that is, 
%the natural transformation
$L\whve R$ is $\tl\tr$-compatible as an $LR$-comodule morphism;

\item $(LR,L\eta R,\ve)$ is a weak comonad and if $\omega:B\to LR(B)$
 is an $\tl\tr$-compatible $LR$-comodule, there is a unique
$\varrho:LR(B)\to B$ in $\K_{\tl\tr}$ making
$(B,\varrho,\omega)$ a Frobenius $(LR,\eta,\whve)$-module,
  given by   
$$  
\varrho : LR(B)\xra{LR(\omega)}LRLR(B)\xra{L\whve_{R(B)}} LR(B)\xra{\ve_B} B;
$$
\item  morphisms between 
$\tl\tr$-compatible $LR$-comodules $(B,\omega)$, $(B',\omega')$  
are $LR$-bimodule morphisms between $(B,\varrho,\omega)$ and
$(B',\varrho',\omega')$.
\end{zlist}
\end{thm}

\begin{proof} (1) By our symmetry assumption,  $\tl R=L\tr$ and the diagram
$$\xymatrix{ LRL \ar[d]_{L\whve R}
\ar[rr]^{L\tr\tl R}_{\tl R L \tr}  
 \ar[drr]_{L\whve R}
          &  & LRLR \ar[d]^{L\whve R}\\
  LR \ar[rr]_{\tl\tr} & & LR }
$$
commutes,  showing  $L\whve R =\tl\tr\cdot L\whve R$.
 
(2) As shown in Proposition \ref{d-equal}, $(B,\omega)$ is $\K_{\tl\tr}$-cofirm and 
hence the existence of $\varrho$ follows by Proposition \ref{equal}.  
 For the Frobenius module $(B,\varrho,\omega)$,  we have the commutative diagram 
$$
\xymatrix{ 
& LRLR(B) \ar[dr]^{LR(\varrho)} \ar[r]^{\quad\ve LR} & LR(B) \ar[dr]^{\varrho}   \\
     LR(B) \ar[r]_{\quad\varrho}   \ar[ru]^{L\eta R} \ar[rd]_{LR(\omega)} & B
 \ar[r]_{\omega\quad} & LR(B)\ar[r]_{\quad\ve_B} & B  \\
  & LRLR(B) \ar[ru]_{L\whve R}& & .
    }
$$
Since $\varrho$ is $\tl R$-compatible, the upper paths yields  
$\varrho\cdot \tl R=\varrho$.
 The lower path is the composite given for $\varrho$. 
\smallskip

(3) Since $\K_{\tl\tr}$ is an ideal class, the assertion about the bimodule morphisms
 follows by Proposition \ref{equal}.
\end{proof}

Instead of $(L,R,\alpha,\beta)$, we may require $(R,L,\talpha, \tbeta)$ to be 
a regular pairing 
(see \ref{pairings}) and relate the  bimodules for $(LR,\eta,\twve)$ with modules
for $(LR,L\twve R)$. By symmetry we obtain:

\begin{thm}\label{reg-basic-d}{\bf Theorem.}
Let $(R,L,\talpha, \tbeta)$ be a regular pairing of functors with $\talpha$ symmetric
 and $\eta:I_\A\to  RL$ any natural transformation. Then,  
\begin{zlist}
\item 
 $\wh\eta := \ttr\ttl\cdot \eta: I_\A  \to RL$ 
is a natural transformation with $\wh{\eta} = \ttr\ttl \cdot\wh{\eta}$ and
 $L\wh{\eta} R = \ttr\ttl \cdot L\wh{\eta} R$, that is, 
$L\wh{\eta} R$ is $\ttr\ttl$-compatible as an $LR$-module morphism 
(see {\rm \ref{comp-morph-mod}});
\item 
$(LR,L\twve R,\tweta)$ is a weak monad and if $\varrho:LR(B)\to B$ is 
 an $\ttl\ttr$-compatible $LR$-module, there is a unique
%$\gamma$-compatible $LR$-comodule 
$\omega:B\to LR(B)$ in $\K_{\ttl\ttr}$ making
$(B,\varrho,\omega)$ a Frobenius $(LR,\wh\eta,\twve)$-bimodule given by   
$$ 
\omega :  B\xra{\tweta_B} LR(B)\xra{L\wh{\eta}_{R(B)}}LRLR(B)\xra{LR(\varrho)} LR(B); 
$$ 
\item  %Any 
 morphisms between 
$\ttl\ttr$-compatible $LR$-modules $(B,\varrho)$, $(B',\varrho')$  
are $(LR,\wh\eta,\twve)$-bimodule morphism between $(B,\varrho,\omega)$ and 
 $(B',\varrho',\omega')$. 
\end{zlist}
\end{thm}

\section{Weak Frobenius monads} \label{weak-frob}

As we have seen in the previous section, for results on the interplay between 
(co)module and bimodule structures for Frobenius monads symmetry 
conditions on our pairings were needed, that is, the intrinsic non-(co)unital (co)monads 
became weak (co)monads. Hence we will concentrate in this section on this kind of  
(co)monads and also apply results from Section \ref{cofirm}. 
 
 \begin{thm} \label{w-frob} {\bf Frobenius property.} \em
Let $(F,m)$ be a non-unital monad, $(F,\delta)$ a non-counital
comonad, $(B,\varrho) \in \uB_F$ and  $(B,\omega) \in \uB^F$.
We say that $(F,m,\delta)$ satisfies the 
{\em Frobenius property} and $(B,\varrho,\omega)$ is a {\em Frobenius bimodule}, 
 provided they induce commutativity of the respective diagrams, 
$$ 
\xymatrix{ & F F F   \ar[rd]^{F m} \\
F F   \ar[ru]^{ \delta F } \ar[rd]_{F \delta}\ar[r]^{\quad m } & 
F  \ar[r]^{\delta \quad} & F F  \\
     & F F F  \ar[ru]_{m F } & ,} \qquad  
\xymatrix{ & F F (B) \ar[rd]^{F (\varrho)} \\
F(B) \ar[ru]^{ \delta_B } \ar[rd]_{F (\omega)}\ar[r]^{\quad \varrho} & 
B  \ar[r]^{\omega \quad} & F (B) \\
     & F F (B) \ar[ru]_{m_B} & .} 
$$
\end{thm}

The Frobenius bimodules as objects and
the morphisms, which are $F$-module as well as $F$-comodule morphisms,  
form a category which we denote by $\uB_F^F$. Transferring the Propositions \ref{equal} and \ref{equal-d} yields:

\begin{thm} \label{w-firm} {\bf Theorem.}
Assume $(F,m,\delta)$ to satisfy the Frobenius property.
Let $(F,\delta,\ve)$ be a weak comonad, $\gamma:=\ve F\cdot \delta$,
and assume $m= \gamma\cdot m$. Then,
\begin{zlist}
\item  for any $\gamma$-compatible $F$-comodule $(B,\omega)$, there is a unique 
$\gamma$-com\-patible $F$-comodule morphism 
$$\varrho: F(B) \xra{F(\omega)} FF(B)\xra{ m_B} F(B) \xra{\ve_B} B$$  
 making $(B,\varrho,\omega)$ a Frobenius bimodule;

\item  any $F$-comodule morphism between $\gamma$-compatible comodules $(B,\omega)$ 
 and $(B',\omega')$ becomes a morphism between the Frobenius bimodules 
 $(B,\varrho,\omega)$ and $(B',\varrho',\omega')$; 

\item there is an isomorphism of categories  
$$\Psi: \ol\B^{F} \to \ol\B^{F}_{F}, \quad (B,\omega) \mapsto (B,\varrho,\omega),$$
with the forgetful functor $U_{F}:\ol\B^{F}_{F} \to \ol\B^{F}$ as inverse, 
where $\ol\B^{F}_{F}$ denotes the category of Frobenius bimodules which are 
$\gamma$-compatible as $F$-comodules.
\end{zlist}
\end{thm}
\begin{proof}
 By our compatibility condition on $m$, we can apply Proposition \ref{reg-basic}
and the formula for $\varrho$ given there.
The assertions about the functors follow directly from the constructions.
\end{proof}

\pagebreak[3]

\begin{thm} \label{w-firm-m} {\bf Theorem.} 
Assume $(F,m,\delta)$ to satisfy the Frobenius property.
Let $(F,m,\eta)$ be a weak monad, $\vartheta:= m \cdot F\eta$,
and assume $\delta = \delta \cdot \vartheta$. Then,
\begin{zlist}
\item   
for a $\vartheta$-compatible $F$-module $(B,\varrho)$, 
there is a unique $\vartheta$-compatible module morphism 
$$\omega:B\xra{\eta_B} F(B)\xra{\delta_B} FF(B)\xra{F(\varrho)} F(B)$$
making $(B,\varrho,\omega)$ a Frobenius bimodule;

\item any $F$-morphism between $\vartheta$-compatible modules 
$(B,\varrho)$, $(B',\varrho')$
becomes a morphism between the Frobenius bimodules
$(B,\varrho,\omega)$ and $(B',\varrho',\omega')$;
\item there is an isomorphism of categories  
$$\Phi: \rB_{F} \to \rB^{F}_{F}, \quad (B,\varrho) \mapsto (B, \varrho,\omega),$$
with the forgetful functor $U^{F}:\rB^{F}_{F} \to \rB_{F}$ as inverse,
where $\rB^{F}_{F}$ denotes the category of Frobenius modules which are 
$\vartheta$-compatible as $F$-modules.
\end{zlist}
\end{thm}
\begin{proof}
By Proposition \ref{reg-basic-d} 
 and the formula for $\omega$ given there.
\end{proof}

\begin{thm} \label{w-Frob-bm} {\bf Definition.} \em
 We call $(F,m,\eta;\delta,\ve)$ a {\em weak Frobenius monad}
provided
  $(F,m,\eta)$ is a weak monad,   $(F,\delta,\ve)$ is a weak comonad,  
 $(F,m,\delta)$ has the Frobenius property  (see (\ref{w-frob})), and
 $m \cdot F\eta = F\ve \cdot\delta$ (i.e. $\vartheta=\gamma$).
\end{thm}

As a first property we observe:

\begin{thm} \label{w-F-F} {\bf Proposition.}  
Let $(F,m,\eta;\delta,\ve)$ be a weak Frobenius monad and assume the idempotent 
$m \cdot F\eta = F\ve \cdot\delta$ to be split by $F\to \ul F\to F$.
Then $\ul F$ has a canonical monad and comonad structure $(\ul F,\ul m,\ul \eta;\ul\delta,\ul \ve)$
which makes it a Frobenius monad.
\end{thm}
\begin{proof} The monad and comonad structures on $\ul F$ are obtained from \ref{q-comod} and \ref{q-mod}
and a routine diagram chase shows that the Frobenius property (see \ref{w-frob}) is satisfied.
\end{proof}
 
 Summarising we obtain our main result for these structures.

\begin{thm}\label{push}{\bf Theorem.}
 Let $(F,m,\eta;\delta,\ve)$ be a weak Frobenius monad. 
Then the constructions in {\rm \ref{w-firm}} and {\rm \ref{w-firm-m}}
 yield category isomorphisms
$$\xymatrix{\ol\B^F \ar[r]^{\Psi\;} & \ol\rB^F_F\ar[r]^{U^F} & \rB_F}, \quad 
 \xymatrix{ \rB_F \ar[r]^{\Phi\;} & \ol\rB^F_F\ar[r]^{U_F} & \ol\B^F}. $$
%$$\Psi:\ol\B^{F} \to\; \ol\rB^{F}_{F} , \quad  \Phi:\rB_{F}\to\; \ol\rB^{F}_{F},$$
where $\ol\rB^{F}_{F}$ denotes the category of those Frobenius $F$-bimodules which are 
 ($\gamma$-)compatible as $F$-comodules and ($\vartheta$-)compatible
as $F$-modules.
\end{thm} 

\begin{proof}
For a weak monad $(F,m,\eta)$, $m$ is $\vartheta$-compatible and 
hence $\gamma$-compatible by our assumption $\gamma=\vartheta$. 
Similarly, $\delta$ is $\vartheta$ compatible and hence the conditions 
in the preceding propositions are satisfied.   
\end{proof}

For (proper) monads and comonads the assertions simplify.
For Proposition \ref{w-firm}
this situation is considered in \cite[Section 4]{BoGo}
and  our results for this case correspond  essentially to 
\cite[Lemma 2, Corollary 1]{BoGo}. 

\pagebreak[3]

\begin{thm}\label{Cor-1}{\bf Corollary.}
 Let $(F,m,\delta)$   satisfy the Frobenius property and assume
$(F,\delta,\ve)$ to be a comonad.  
\begin{zlist} 
\item 
For any counital $F$-comodule $\omega:B\to F(B)$, there is some $F$-module morphism
$\varrho:F(B)\to B$  making
$(B,\varrho,\omega)$ a Frobenius bimodule.  
\item If $(F,m)$ allows for a unit, then $(B,\varrho)$ is a unital $F$-module.      
\item If $m\cdot\delta =I_F$, then, for any Frobenius bimodule 
$(B,\varrho,\omega)$, $(B,\varrho)$ is a firm $F$-module.
\end{zlist}
\end{thm}

\begin{proof}   
 (1), (2) hold by Theorem \ref{w-firm}; (3) follows from Theorem \ref{prop-w-sep}. 
\end{proof}

\begin{thm}\label{Cor-2}{\bf Corollary.}
 Let $(F,m,\delta)$   satisfy the Frobenius property and assume
 $(F,m,\eta)$ to be a monad.
\begin{zlist} 
\item 
For any  unital $F$-module $\varrho:F(B)\to B$, there is some $F$-comodule morphism
$\omega: B\to F(B)$ (given in {\rm \ref{reg-basic-d}}) making
$(B,\varrho,\omega)$ a Frobenius bimodule.  
\item If $(F,\delta)$ allows for a counit, then $(B,\omega)$ is a counital $F$-comodule.   
\item If $m\cdot\delta=I_F$, then, for any Frobenius bimodule $(B,\varrho,\omega)$,
$(B,\omega)$
 is a firm $F$-comodule.
\end{zlist}
\end{thm}

For proper monads and comonads  $F$, all non-unital $F$-modules are compatible
and all non-counital $F$-comodules are compatible, that is, 
$\uB_{F}=\ul\B_{F}$ and  $\uB^{F}= \ol\B^{F}$. Thus we have:

\btm \label{Cor}  {\bf Corollary.} 
Let $(F,m,\eta;\delta,\ve)$ be a Frobenius monad.
  There are category  isomorphisms
$$\Psi: \uB^{F} \to\; \uB^{F}_{F} , \quad  \Phi:\uB_{F}\to\; \uB^{F}_{F},$$
where $ \uB^{F}_{F}$ denotes the category of non-unital and non-counital Frobenius 
$F$-bimodules,  and 
$$\Psi': \B^{F} \to\; \B^{F}_{F} , \quad  \Phi':\B_{F}\to\; \B^{F}_{F},$$
where $ \B^{F}_{F}$ is the category of unital and counital Frobenius 
$F$-bimodules.  
\etm

It is easy to see that (by (co)restriction) these isomorphisms induce isomorphisms
between the category of unital $F$-modules, counital $F$-comodules, and of unital
and counital Frobenius bimodules, an observation following from Eilenberg-Moore 
 [4], which may be considered as the starting point for the categorical 
treatment of Frobenius algebras.

\end{document}